\documentclass[10pt,reqno,a4paper,oneside,11pt]{amsart}%
\usepackage{tikz}
\usepackage{mathtools}
\usepackage{extarrows}
\usetikzlibrary{arrows}
\usepackage{xparse}
\usepackage{amsfonts}
\usepackage{amsfonts}
\usepackage{amsfonts}
\usepackage{amsfonts}
\usepackage{amsfonts}
\usepackage{amsfonts}
\usepackage{amsfonts}
\usepackage{amsfonts}
\usepackage{amsfonts}
\usepackage{amsfonts}
\usepackage{amsfonts}
\usepackage{amsfonts}
\usepackage{amsfonts}
\usepackage{amsfonts}
\usepackage{amsfonts}
\usepackage{amsfonts}
\usepackage{amsfonts}
\usepackage{amsfonts}
\usepackage{amsfonts}
\usepackage{amsfonts}
\usepackage{mathrsfs}
\usepackage{mathrsfs}
\usepackage{amsfonts}
\usepackage{amssymb}
\usepackage{amsmath}
\usepackage{amsthm}
\usepackage{graphicx}%
\providecommand{\U}[1]{\protect\rule{.1in}{.1in}}
\newtheorem{theorem}{Theorem}[section]

\newtheorem{lemma}[theorem]{Lemma}

\newtheorem{definition}{Definition}[section]

\newtheorem{property}{Property}[section]

\theoremstyle{definition}
\theoremstyle{remark}
\numberwithin{equation}{section}

\ifx\pdfoutput\relax\let\pdfoutput=\undefined\fi
\newcount\msipdfoutput
\ifx\pdfoutput\undefined\else
\ifcase\pdfoutput\else
\ifx\paperwidth\undefined\else
\ifdim\paperheight=0pt\relax\else\pdfpageheight\paperheight\fi
\ifdim\paperwidth=0pt\relax\else\pdfpagewidth\paperwidth\fi
\fi\fi\fi
\begin{document}
\pagestyle{myheadings}

\begin{center}
{\huge \textbf{A system of $k$ Sylvester-type quaternion matrix equations with $3k+1$ variables}}\footnote{This research was supported by  the National Natural
Science Foundation of China (Grant no. 11971294). Email address: wqw@t.shu.edu.cn}

\bigskip

{ \textbf{Qing-Wen Wang$^{a,*}$, Mengyan Xie$^{a}$, }}

{\small
\vspace{0.25cm}

$a.$ Department of Mathematics, Shanghai University, Shanghai 200444, P. R. China}

\end{center}

\begin{quotation}
\noindent\textbf{Abstract:} In this paper, we provide some solvability conditions in terms of ranks for the existence of  a general solution to a system of $k$ Sylvester-type quaternion matrix equations with $3k+1$ variables
$A_{i}X_{i}+Y_{i}B_{i}+C_{i}Z_{i}D_{i}+F_{i}Z_{i+1}G_{i}=E_{i},~i=\overline{1,k}$. As applications of this system, we present rank equalities as the necessary and sufficient conditions  for the existence of a  general solution
to some systems of quaternion matrix equations
$A_{i}X_{i}+(A_{i}X_{i})_{\phi}+C_{i}Z_{i}(C_{i})_{\phi}+F_{i}Z_{i+1}(F_{i})_{\phi}=E_{i},~i=\overline{1,k}$.
 involving $\phi$-Hermicity.  
\newline%
\end{quotation}

\section{\textbf{Introduction}}

Sylvester matrix equation and its generalizations are closely related with many problems in robust control \cite{Varga}, neural network \cite{zhangyong}, output feedback control \cite{VLS}, graph theory \cite{Dmytryshyn}, and so on. Recently, some researchers have considered Sylvester-type matrix equations over the quaternion algebra. Quaternion algebra is an associative and noncommutative division algebra over the real number field. Nowadays quaternion has applications in signal
and color image processing, quantum physics, and so on.
%and have been investigated by several researchers (e.g., \cite{jkb}-\cite{xibanya01}, \cite{Flanders}, \cite{Hajarian001}-\cite{Hajarian003}, \cite{heela2019}, \cite{bims}, \cite{jcam2018}-\cite{heac2017}, \cite{mao16}, \cite{Roth}, \cite{Q.W12}-\cite{Wuaiguo01}, \cite{zhoubin001}).

In this paper, we aim to consider the solvable conditions for the existence of a general solution to the systems of quaternion matrix equations:
\begin{align}
\label{mainsystem01}
  \left\{\begin{array}{c}
A_{1}X_{1}+Y_{1}B_{1}+C_{1}Z_{1}D_{1}+F_{1}Z_{2}G_{1}=E_{1} \\
A_{2}X_{2}+Y_{2}B_{2}+C_{2}Z_{2}D_{2}+F_{2}Z_{3}G_{2}=E_{2}\\
\vdots\\
A_{k}X_{k}+Y_{k}B_{k}+C_{k}Z_{k}D_{k}+F_{k}Z_{k+1}G_{k}=E_{k},
\end{array}
\right.
\end{align}
where $A_{k},B_{k}, C_{k},D_{k},F_{k},G_{k}$ and $E_{k}$ are given matrices, $X_{k},~Y_{k},~Z_{k},~Z_{k+1}$ are unknowns. We give a practical necessary and sufficient conditions for the existence of a solution to the system (\ref{mainsystem01}) in terms of ranks. The system (\ref{mainsystem01}) have been firstly investigated by \cite{wanghe4444444} when $k=1$. He \cite{heaaca02} give the solvable conditions and general solution to system (\ref{mainsystem01}) when $k=2$. However, it is hard to solve the Theorem \ref{theorem31} by using the same approach in (\ref{mainsystem01}).  We will use another approach to prove Theorem \ref{theorem31}.

Since Rodman \cite{rodman} first presented the  definition of $\phi$-Hermitian quaternion matrix, there have been several papers to discuss the quaternion matrix equations involving $\phi$-Hermicity (e.g., \cite{Zhuohengfilomat}, \cite{heaaca}, \cite{heaaca02}, \cite{heela}). As applications of Theorem \ref{theorem31}, we present rank equalities as the solvable conditions for the system of quaternion matrix equations involving $\phi$-Hermicity
\begin{align}\label{mainsystem02}
  \left\{\begin{array}{c}
A_{1}X_{1}+(A_{1}X_{1})_{\phi}+C_{1}Z_{1}(C_{1})_{\phi}+F_{1}Z_{2}(F_{1})_{\phi}=E_{1},\\
A_{2}X_{2}+(A_{2}X_{2})_{\phi}+C_{2}Z_{2}(C_{2})_{\phi}+F_{2}Z_{3}(F_{3})_{\phi}=E_{2},\\
\vdots\\
A_{k}X_{k}+(A_{k}X_{k})_{\phi}+C_{k}Z_{k}(C_{k})_{\phi}+F_{k}Z_{k+1}(F_{k})_{\phi}=E_{k},\\
Z_{i}=(Z_{i})_{\phi}.
\end{array}
  \right.
\end{align}

The remainder of this paper is organized as follows. In Section 2, we review some definitions of quaternion algebra and consider a quaternion matrix equation. In Section 3, we present some practical necessary and sufficient conditions for the existence of a solution to the system (\ref{mainsystem01}) in terms of ranks and Moore-Penrose inverses.  In Section 4, we review the definition of $\phi$-Hermitian quaternion matrix and present some solvability conditions to the system (\ref{mainsystem02}) involving $\phi$-Hermicity . 

\section{\textbf{Preliminaries}}

Let $\mathbb{R}$ and $\mathbb{H}^{m\times n}$ stand, respectively, for the real field and the space of all $m\times n$ matrices over the real quaternion algebra
\[
\mathbb{H}%
=\big\{a_{0}+a_{1}\mathbf{i}+a_{2}\mathbf{j}+a_{3}\mathbf{k}\big|~\mathbf{i}^{2}=\mathbf{j}^{2}=\mathbf{k}^{2}=\mathbf{ijk}=-1,a_{0}%
,a_{1},a_{2},a_{3}\in\mathbb{R}\big\}.
\]
 The symbols $r(A)$ and $A^{\ast}$ stand for the rank of a given quaternion matrix $A$ and the conjugate transpose of $A$ and the transposed of $A$, respectively. $I$ and $0$ are the identity matrix and zero matrix with appropriate sizes, respectively. The Moore-Penrose inverse
$A^{\dag}$ of a quaternion matrix $A$, is defined to be the unique matrix $A^{\dag},$
such that
\begin{align*}
\text{(i)}~AA^{\dag}A=A,~\text{(ii)}~A^{\dag}AA^{\dag}=A^{\dag},~\text{(iii)}%
~(AA^{\dag})^{\ast}=AA^{\dag},~\text{(iv)}~(A^{\dag}A)^{\ast}=A^{\dag}A.
\end{align*}Furthermore, $L_{A}$ and $R_{A}$ stand for the projectors $L_{A}%
=I-A^{\dag}A$ and $R_{A}=I-AA^{\dag}$ induced by $A$, respectively. For more definitions and properties of quaternions, we refer the reader to the book \cite{rodman}.

In order to solve the system (\ref{mainsystem01}), we need the solvability conditions and general solution to the quaternion matrix equation
\begin{align}\label{july25equ001}
A_{1}X_{1}+Y_{1}B_{1}+C_{1}Z_{1}D_{1}+F_{1}Z_{2}G_{1}=E_{1},
\end{align}where $A_{1},B_{1},C_{1},D_{1},F_{1},G_{1},E_{1}$ are given matrices, and $X_{1},Y_{1},Z_{1},Z_{2}$ are unknowns.

\begin{lemma}\label{lemma01} (\cite{hewang4}, \cite{wanghe4444444}).  Consider the quaternion matrix equation (\ref{july25equ001}). Set
\begin{align*}
A_{11}=R_{A_{1}}C_{1},~B_{11}=D_{1}L_{B_{1}},~C_{11}=R_{A_{1}}F_{1},~D_{11}=G_{1}L_{B_{1}},
\end{align*}
\begin{align*}
E_{11}=R_{A_{1}}E_{1}L_{B_{1}},~M_{11}=R_{A_{11}}C_{11},~N_{11}=D_{11}L_{B_{11}},~S_{11}=C_{11}L_{M_{11}}.
\end{align*}
Then the equation
(\ref{july25equ001}) is consistent if and only if
\begin{align*}
R_{M_{11}}R_{A_{11}}E_{11}=0,  ~E_{11}L_{B_{11}}L_{N_{11}}=0, ~R_{A_{11}}E_{11}L_{D_{11}}=0,~  R_{C_{11}}E_{11}L_{B_{11}}=0.
\end{align*}
In this case, the general solution to (\ref{july25equ001}) can be
expressed as
\begin{align*}
X_{1}=A_{1}^{\dag}(E_{1}-C_{1}Z_{1}D_{1}-F_{1}Z_{2}G_{1})-T_{17}B_{1}+L_{A_{1}}T_{16},
\end{align*}
\begin{align*}
Y_{1}=R_{A_{1}}(E_{1}-C_{1}Z_{1}D_{1}-F_{1}Z_{2}G_{1})B_{1}^{\dag}%
+A_{1}T_{17}+T_{18}R_{B_{1}},
\end{align*}
\begin{align*}
Z_{1}=&
A_{11}^{\dag}E_{11}B_{11}^{\dag}-A_{11}^{\dag}C_{11}M_{11}^{\dag}E_{11}B_{11}^{\dag}-A_{11}^{\dag}S_{11}C_{11}^{\dag
}E_{11}N_{11}^{\dag}D_{11}B_{11}^{\dag}-A_{11}^{\dag}S_{11}T_{12}R_{N_{11}}D_{11}B_{11}^{\dag}\\&+L_{A_{11}}T_{14}+T_{15}
R_{B_{11}},
\end{align*}
\begin{align*}
Z_{2}=M_{11}^{\dag}E_{11}D_{11}^{\dag}+S_{11}^{\dag}S_{11}C_{11}^{\dag}E_{11}N_{11}^{\dag}+L_{M_{11}}L_{S_{11}}T_{11}%
+L_{M_{11}}T_{12}R_{N_{11}}+T_{13}R_{D_{11}},
\end{align*}
where $T_{11},\ldots,T_{18}$ are arbitrary matrices over $\mathbb{H}$
with appropriate sizes.
\end{lemma}

The following lemma is due to Marsaglia and Styan \cite{GPH}, which can be
generalized to the quaternion algebra.

\begin{lemma}
\label{hlemma03}(\cite{GPH}). Let $A\in\mathbb{H}^{m\times n},B\in
\mathbb{H}^{m\times k},$ and $C\in\mathbb{H}^{l\times
n},$ be given. Then\newline
$(1)~ r(A)+r(R_{A}B)=r(B)+r(R_{B}A)=r\begin{pmatrix}A&B\end{pmatrix}.$\newline
$ (2)~ r(A)+r(CL_{A})=r(C)+r(AL_{C})=r%
\begin{pmatrix}
A\\
C
\end{pmatrix}
.$
\end{lemma}

\section{\textbf{Solvability conditions to the system (\ref{mainsystem01})}}
The goal of this section is to consider the solvability conditions to the system (\ref{mainsystem01}) in terms of rank equalities. 
\begin{theorem}\label{theorem31}
	The system (\ref{mainsystem01}) is consistent if and only if the following rank equalities hold for all $i=\overline{1,k},~1\le m\le n\le k$:
\begin{align}\label{july25equ002}
r\begin{pmatrix}E_{i}&A_{i}&C_{i}&F_{i}\\B_{i}&0&0&0\end{pmatrix}=r\begin{pmatrix}A_{i}&C_{i}&F_{i}\end{pmatrix}+r(B_{i}),~
r\begin{pmatrix}E_{i}&A_{i}\\B_{i}&0\\D_{i}&0\\G_{i}&0\end{pmatrix}=r\begin{pmatrix}B_{i}\\D_{i}\\G_{i}\end{pmatrix}+r(A_{i}),
\end{align}
\begin{align}\label{july25equ003}
r\begin{pmatrix}E_{i}&A_{i}&C_{i}\\B_{i}&0&0\\G_{i}&0&0\end{pmatrix}=r\begin{pmatrix}A_{i}&C_{i}\end{pmatrix}+r\begin{pmatrix}B_{i}\\G_{i}\end{pmatrix},~
r\begin{pmatrix}E_{i}&A_{i}&F_{i}\\B_{i}&0&0\\D_{i}&0&0\end{pmatrix}=r\begin{pmatrix}A_{i}&F_{i}\end{pmatrix}+r\begin{pmatrix}B_{i}\\D_{i}\end{pmatrix},
\end{align}
\begin{align}\label{july25equ004}
&r\begin{pmatrix}
\begin{smallmatrix}
C_{m}&E_{m}&F_{m}& & & & &A_{m}& & & \\
&G_{m}& &D_{m+1}& & & & & & & \\
& &C_{m+1}&-E_{m+1}&F_{m+1}& & & &A_{m+1}& & \\
& & &\ddots&\ddots&\ddots & & & &\ddots& \\
& & & &C_{n}&(-1)^{n-m}E_{n}&F_{n}& & & &A_{n}\\
&B_{m}& & & & & & & & & \\
& & &B_{m+1}& & & & & & & \\
& & & &\ddots & & & & & & \\
& & & & &B_{n}& & & & & 
\end{smallmatrix}
\end{pmatrix}\nonumber\\
&=r\begin{pmatrix}
\begin{smallmatrix}
C_{m}&F_{m}& & &A_{m}& & & \\
&C_{m+1}&F_{m+1}& & &A_{m+1}& & \\
& &\ddots&\ddots & & &\ddots & \\
& & &C_n &F_n& & &A_n
\end{smallmatrix}
\end{pmatrix}
+r\begin{pmatrix}
\begin{smallmatrix}
G_{m}&D_{m+1}& & & \\
&G_{m+1}&D_{m+2}& & \\
& &\ddots&\ddots& \\
& & &G_{n-1}&D_{n}\\
B_{m}& & & & \\
&B_{m+1}& & & \\
& & &\ddots& \\
& & & &B_{n}
\end{smallmatrix}
\end{pmatrix},
\end{align}

\begin{align}\label{july25equ005}
&r\begin{pmatrix}
\begin{smallmatrix}
D_{m}& & & & & & & & & \\
E_{m}&F_{m}& & & & &A_m& & & \\
G_{m}& &D_{m+1}& & & & & & & \\
&C_{m+1}&-E_{m+1}&F_{m+1}& & &&A_{m+1}&&\\
& &G_{m+1}& &\ddots && & &\ddots &\\
& & &C_{m+2}&\ddots &D_n& & &&A_n\\
& & & &\ddots &(-1)^{n-m}E_n& & &&\\
& & & & &G_n& & &&\\
B_{m}& & & & & & & &&\\
& &B_{m+1}& & & & & &&\\
& & &  &\ddots & & & &&\\
& & &  & &B_n& & &&
\end{smallmatrix}
\end{pmatrix}\nonumber\\
&=r\begin{pmatrix}
\begin{smallmatrix}
F_{m}& & & &A_m&&\\
C_{m+1}&F_{m+1}& & & &A_{m+1}&\\
&C_{m+2}&\ddots&F_{n-1}& &&\ddots \\
& &\ddots&C_n& &&&A_n
\end{smallmatrix}
\end{pmatrix}
+r\begin{pmatrix}
\begin{smallmatrix}
D_{m}& & \\
G_{m}&D_{m+1}& \\
&G_{m+1}&\ddots&\\
& &\ddots &D_n\\
& & &G_n\\
B_{m}& & \\
&B_{m+1}& \\
& &\ddots\\
& & &B_n
\end{smallmatrix}
\end{pmatrix},
\end{align}
\begin{align}\label{july25equ006}
&r\begin{pmatrix}
\begin{smallmatrix}
C_{m}&E_{m}&F_{m}& & & & &A_m& &\\
&G_{m}& &D_{m+1}& & & & & &\\
& &C_{m+1}&-E_{m+1}&F_{m+1}& & & &A_{m+1}&\\
& & &G_{m+1}& &D_{m+2}& & & &\\
&& & &\ddots &\ddots &\ddots & &  &\ddots \\
& & & & & &D_n& & & \\
& & & & &C_n&(-1)^{n-m}E_n& & &&A_n\\
& & & & & &G_n& & &\\
&B_{m}& & & & & & & &\\
& & &B_{m+1}& & & & & &\\
& & & & &\ddots & & & &\\
& & & & & &B_n& & &
\end{smallmatrix}
\end{pmatrix}\nonumber
\end{align}
\begin{align}
&=
r\begin{pmatrix}
\begin{smallmatrix}
C_{m}&F_{m}& &&A_{m}& & \\
&C_{m+1}&F_{m+1}& &&A_{m+1}& \\
& &\ddots  & & &&\ddots&\\
& &  &C_n&&&&A_{n}\\
\end{smallmatrix}
\end{pmatrix}
+r\begin{pmatrix}
\begin{smallmatrix}
G_{m}&D_{m+1}& \\
&G_{m+1}&D_{m+2}\\
& &\ddots \\
& &\ddots &D_n\\
& & &G_n\\
B_{m}& & \\
&B_{m+1}& \\
& &\ddots\\
& &&B_n
\end{smallmatrix}
\end{pmatrix},
\end{align}

\begin{align}\label{july25equ007}
&r\begin{pmatrix}
\begin{smallmatrix}
D_{m}& & & & & & & & \\
E_{m}&F_{m}& & & & &&A_{m}& & \\
G_{m}& &D_{m+1}& & & & & & \\
&C_{m+1}&-E_{m+1}&F_{m+1}& & & &&A_{m+1}& \\
& &G_{m+1}& &D_{3}& & & &&\ddots \\
& &&\ddots &\ddots &\ddots & && & \\
& & & &C_{n}&(-1)^{n-m}E_{n}&F_n& &&&A_{n}\\
B_{m}& & & & & & & & \\
& &B_{m+1}& & & & & & \\
& &  &\ddots && & & & \\
& & & & &B_n& & & 
\end{smallmatrix}
\end{pmatrix}\nonumber\\
&=r\begin{pmatrix}
\begin{smallmatrix}
F_{m}& & &&A_{m}& & \\
C_{m+1}&F_{m+1}& & &&A_{m+1}& \\
&\ddots &\ddots & & &&\ddots \\
& &C_{n}&F_n& &&&A_{n}
\end{smallmatrix}
\end{pmatrix}
+r\begin{pmatrix}
\begin{smallmatrix}
D_{m}& & \\
G_{m}&D_{m+1}& \\
&G_{m+1}&D_{m+2}\\
&&\ddots&\ddots \\
&&&G_{n}&D_{n}\\
B_{m}& & \\
&B_{m+1}& \\
&&\ddots \\
& &&&B_{n}
\end{smallmatrix}
\end{pmatrix},
\end{align}
where the blank entries in above rank equalities are all zeros.
\end{theorem}

\begin{proof}
The proof is by mathematical induction on $n$. For $n=1$, the statement is clear. We assume that the
(\ref{july25equ002})--(\ref{july25equ007}) is true for $n=k-1$. So next we will show that the induction is true  on $n=k$.

We divide the system (\ref{mainsystem01}) into the following $k$ equations:
\begin{align}
A_{1}X_{1}&+Y_{1}B_{1}+C_{1}Z_{1}D_{1}+F_{1}Z_{2}G_{1}=E_{1}, \label{july25equ008}\\
A_{2}X_{2}&+Y_{2}B_{2}+C_{2}Z_{2}D_{2}+F_{2}Z_{3}G_{2}=E_{2},\label{july25equ009}\\
&\vdots\nonumber\\
A_{k}X_{k}&+Y_{k}B_{k}+C_{k}Z_{k}D_{k}+F_{k}Z_{k+1}G_{k}=E_{k}.\label{july25equ010}
\end{align}
Applying Lemma \ref{lemma01}, the equations (\ref{july25equ008})--(\ref{july25equ010}) are consistent if and only if rank equalities (\ref{july25equ002}),  (\ref{july25equ003}) hold. Put $A_{ii}=R_{A_{i}}C_{i},~B_{ii}=D_{i}L_{B_{i}},~C_{ii}=R_{A_{i}}F_{i},~D_{ii}=G_{i}L_{B_{i}},~E_{ii}=R_{A_{i}}E_{i}L_{B_{i}},~M_{ii}=R_{A_{ii}}C_{ii},~N_{ii}=D_{ii}L_{B_{ii}},~S_{ii}=C_{ii}L_{M_{ii}},~(i=\overline{1,k}),$ the general solution of equation 
\begin{equation}\label{july25equ011}
A_{i}X_{i}+Y_{i}B_{i}+C_{i}Z_{i}D_{i}+F_{i}Z_{i+1}G_{i}=E_{i}
\end{equation}
can be expressed as
\begin{align*}
X_{ i}=A_{ i}^{\dag}(E_{ i}-C_{ i}Z_{ i}D_{ i}-F_{ i}Z_{ i+ i}G_{ i})-T_{ i7}B_{ i}+L_{A_{ i}}T_{ i6},
\end{align*}
\begin{align*}
Y_{ i}=R_{A_{ i}}(E_{ i}-C_{ i}Z_{ i}D_{ i}-F_{ i}Z_{ i+ i}G_{ i})B_{ i}^{\dag}%
+A_{ i}T_{ i7}+T_{ i8}R_{B_{ i}},
\end{align*}
\begin{align}\label{july25equ012}
Z_{ i}=&
A_{ i i}^{\dag}E_{ i i}B_{ i i}^{\dag}-A_{ i i}^{\dag}C_{ i i}M_{ i i}^{\dag}E_{ i i}B_{ i i}^{\dag}-A_{ i i}^{\dag}S_{ i i}C_{ i i}^{\dag
}E_{ i i}N_{ i i}^{\dag}D_{ i i}B_{ i i}^{\dag}-A_{ i i}^{\dag}S_{ i i}T_{ i 2}R_{N_{ i i}}D_{ i i}B_{ i i}^{\dag}\nonumber\\&+L_{A_{ i i}}T_{ i4}+T_{ i5}
R_{B_{ i i}},
\end{align}
\begin{align}\label{july25equ013}
Z_{ i+ 1}=M_{ i i}^{\dag}E_{ i i}D_{ i i}^{\dag}+S_{ i i}^{\dag}S_{ i i}C_{ i i}^{\dag}E_{ i i}N_{ i i}^{\dag}+L_{M_{ i i}}L_{S_{ i i}}T_{ i 1}%
+L_{M_{ i i}}T_{ i 2}R_{N_{ i i}}+T_{ i3}R_{D_{ i i}},
\end{align}
where $T_{i1},\ldots,T_{i8}$ are arbitrary matrices over $\mathbb{H}$
with appropriate sizes.

Let $Z_i$ in the $(i+1)$th equation be equal to $Z_{i+1}$ in the $i$th equation, $0\le i\le k$ we can obtain the following system
 \begin{equation}\label{july25equ014}
\left\{\begin{array}{c}
\widehat{A}_{1}\begin{pmatrix}T_{11}\\T_{24}\end{pmatrix}+\begin{pmatrix}T_{13}&T_{25}\end{pmatrix}\widehat{B}_{1}
+\widehat{C}_{1}T_{12}\widehat{D}_{1}+\widehat{F}_{1}T_{22}\widehat{G}_{1}=\widehat{E}_{1},\\
\widehat{A}_{2}\begin{pmatrix}T_{21}\\T_{34}\end{pmatrix}+\begin{pmatrix}T_{23}&T_{35}\end{pmatrix}\widehat{B}_{2}
+\widehat{C}_{2}T_{22}\widehat{D}_{2}+\widehat{F}_{2}T_{32}\widehat{G}_{2}=\widehat{E}_{2},\\
\vdots\\
\widehat{A}_{k-1}\begin{pmatrix}T_{k-1,1}\\T_{k4}\end{pmatrix}+\begin{pmatrix}T_{k-1,3}&T_{k5}\end{pmatrix}\widehat{B}_{k-1}
+\widehat{C}_{k-1}T_{k-1,2}\widehat{D}_{k-1}+\widehat{F}_{k-1}T_{k2}\widehat{G}_{k-1}=\widehat{E}_{k-1},\\
\end{array}
\right.
\end{equation}where 
\begin{align}\label{july25equ015}
\widehat{A}_{k-1}=&\begin{pmatrix}L_{M_{k-1,k-1}}L_{S_{k-1,k-1}}&-L_{A_{kk}}\end{pmatrix},~
\widehat{B}_{k-1}=\begin{pmatrix}R_{D_{k-1,k-1}}\\-R_{B_{kk}}\end{pmatrix},
\\\widehat{C}_{k-1}=&L_{M_{k-1,k-1}},~\widehat{D}_{k-1}=R_{N_{k-1,k-1}},\widehat{F}_{k-1}=A_{kk}^{\dag}S_{kk},~\widehat{G}_{k-1}=R_{N_{kk}}D_{kk}B_{kk}^{\dag},
\end{align}
\begin{align} \label{july25equ016}
\widehat{E}_{k-1}=&A_{kk}^{\dag}E_{kk}B_{kk}^{\dag}-A_{kk}^{\dag}C_{kk}M_{kk}^{\dag}E_{kk}B_{kk}^{\dag}-A_{kk}^{\dag}S_{kk}C_{kk}^{\dag
}E_{kk}N_{kk}^{\dag}D_{kk}B_{kk}^{\dag}\nonumber\\
&-M_{k-1,k-1}^{\dag}E_{k-1,k-1}D_{k-1,k-1}^{\dag}-S_{k-1,k-1}^{\dag}S_{k-1,k-1}C_{k-1,k-1}^{\dag}E_{k-1,k-1}N_{k-1,k-1}^{\dag}.
\end{align}
Note that the system (\ref{july25equ014}) is the same as (\ref{mainsystem01}) when $n=k-1$, so  we can apply the induction to (\ref{july25equ014}). The system (\ref{july25equ014}) is consistent if and only if 
\begin{align}\label{july25equ017}
r\begin{pmatrix}\widehat{E}_{i}&\widehat{A}_{i}&\widehat{C}_{i}&\widehat{F}_{i}\\\widehat{B}_{i}&0&0&0\end{pmatrix}=r\begin{pmatrix}\widehat{A}_{i}&\widehat{C}_{i}&\widehat{F}_{i}\end{pmatrix}+r(\widehat{B}_{i}),~
r\begin{pmatrix}\widehat{E}_{i}&\widehat{A}_{i}\\\widehat{B}_{i}&0\\\widehat{D}_{i}&0\\\widehat{G}_{i}&0\end{pmatrix}=r\begin{pmatrix}\widehat{B}_{i}\\\widehat{D}_{i}\\\widehat{G}_{i}\end{pmatrix}+r(\widehat{A}_{i}),
\end{align}
\begin{align}\label{july25equ018}
r\begin{pmatrix}\widehat{E}_{i}&\widehat{A}_{i}&\widehat{C}_{i}\\\widehat{B}_{i}&0&0\\\widehat{G}_{i}&0&0\end{pmatrix}=r\begin{pmatrix}\widehat{A}_{i}&\widehat{C}_{i}\end{pmatrix}+r\begin{pmatrix}\widehat{B}_{i}\\\widehat{G}_{i}\end{pmatrix},~
r\begin{pmatrix}\widehat{E}_{i}&\widehat{A}_{i}&\widehat{F}_{i}\\\widehat{B}_{i}&0&0\\\widehat{D}_{i}&0&0\end{pmatrix}=r\begin{pmatrix}\widehat{A}_{i}&\widehat{F}_{i}\end{pmatrix}+r\begin{pmatrix}\widehat{B}_{i}\\\widehat{D}_{i}\end{pmatrix},
\end{align}
\begin{align}\label{july25equ019}
&r\begin{pmatrix}
\begin{smallmatrix}
\widehat{C}_{m}&\widehat{E}_{m}&\widehat{F}_{m}& & & & &\widehat{A}_{m}& & & \\
&\widehat{G}_{m}& &\widehat{D}_{m+1}& & & & & & & \\
& &\widehat{C}_{m+1}&-\widehat{E}_{m+1}&\widehat{F}_{m+1}& & & &\widehat{A}_{m+1}& & \\
& & &\ddots&\ddots&\ddots & & & &\ddots& \\
& & & &\widehat{C}_{n}&(-1)^{n-m}\widehat{E}_{n}&\widehat{F}_{n}& & & &\widehat{A}_{n}\\
&\widehat{B}_{m}& & & & & & & & & \\
& & &\widehat{B}_{m+1}& & & & & & & \\
& & & &\ddots & & & & & & \\
& & & & &\widehat{B}_{n}& & & & & 
\end{smallmatrix}
\end{pmatrix}\nonumber\\
&=r\begin{pmatrix}
\begin{smallmatrix}
\widehat{C}_{m}&\widehat{F}_{m}& & &\widehat{A}_{m}& & & \\
&\widehat{C}_{m+1}&\widehat{F}_{m+1}& & &\widehat{A}_{m+1}& & \\
& &\ddots&\ddots & & &\ddots & \\
& & &\widehat{C}_n &\widehat{F}_n& & &\widehat{A}_n
\end{smallmatrix}
\end{pmatrix}
+r\begin{pmatrix}
\begin{smallmatrix}
\widehat{G}_{m}&\widehat{D}_{m+1}& & & \\
&\widehat{G}_{m+1}&\widehat{D}_{m+2}& & \\
& &\ddots&\ddots& \\
& & &\widehat{G}_{n-2}&\widehat{D}_{n}\\
\widehat{B}_{m}& & & & \\
&\widehat{B}_{m+1}& & & \\
& & &\ddots& \\
& & & &\widehat{B}_{n}
\end{smallmatrix}
\end{pmatrix},
\end{align}

\begin{align}\label{july25equ020}
&r\begin{pmatrix}
\begin{smallmatrix}
\widehat{D}_{m}& & & & & & & & & \\
\widehat{E}_{m}&\widehat{F}_{m}& & & & &\widehat{A}_m& & & \\
\widehat{G}_{m}& &\widehat{D}_{m+1}& & & & & & & \\
&\widehat{C}_{m+1}&-\widehat{E}_{m+1}&\widehat{F}_{m+1}& & &&\widehat{A}_{m+1}&&\\
& &\widehat{G}_{m+1}& &\ddots && & &\ddots &\\
& & &\widehat{C}_{m+2}&\ddots &\widehat{D}_n& & &&\widehat{A}_n\\
& & & &\ddots &(-1)^{n-m}\widehat{E}_n& & &&\\
& & & & &\widehat{G}_n& & &&\\
\widehat{B}_{m}& & & & & & & &&\\
& &\widehat{B}_{m+1}& & & & & &&\\
& & &  &\ddots & & & &&\\
& & &  & &\widehat{B}_n& & &&
\end{smallmatrix}
\end{pmatrix}\nonumber\\
&=r\begin{pmatrix}
\begin{smallmatrix}
\widehat{F}_{m}& & & &\widehat{A}_m&&\\
\widehat{C}_{m+1}&\widehat{F}_{m+1}& & & &\widehat{A}_{m+1}&\\
&\widehat{C}_{m+2}&\ddots&\widehat{F}_{n-1}& &&\ddots \\
& &\ddots&\widehat{C}_n& &&&\widehat{A}_n
\end{smallmatrix}
\end{pmatrix}
+r\begin{pmatrix}
\begin{smallmatrix}
\widehat{D}_{m}& & \\
\widehat{G}_{m}&\widehat{D}_{m+1}& \\
&\widehat{G}_{m+1}&\ddots&\\
& &\ddots &\widehat{D}_n\\
& & &\widehat{G}_n\\
\widehat{B}_{m}& & \\
&\widehat{B}_{m+1}& \\
& &\ddots\\
& & &\widehat{B}_n
\end{smallmatrix}
\end{pmatrix},
\end{align}
\begin{align}\label{july25equ021}
&r\begin{pmatrix}
\begin{smallmatrix}
\widehat{C}_{m}&\widehat{E}_{m}&\widehat{F}_{m}& & & & &\widehat{A}_1& &\\
&\widehat{G}_{m}& &\widehat{D}_{m+1}& & & & & &\\
& &\widehat{C}_{m+1}&-\widehat{E}_{m+1}&\widehat{F}_{m+1}& & & &\widehat{A}_2&\\
& & &\widehat{G}_{m+1}& &\widehat{D}_{m+2}& & & &\\
&& & &\ddots &\ddots &\ddots & &  &\ddots \\
& & & & & &\widehat{D}_n& & & \\
& & & & &\widehat{C}_n&(-1)^{n-m}\widehat{E}_n& & &&\widehat{A}_n\\
& & & & & &\widehat{G}_n& & &\\
&\widehat{B}_{m}& & & & & & & &\\
& & &\widehat{B}_{m+1}& & & & & &\\
& & & & &\ddots & & & &\\
& & & & & &\widehat{B}_n& & &
\end{smallmatrix}
\end{pmatrix}\nonumber\\
&=
r\begin{pmatrix}
\begin{smallmatrix}
\widehat{C}_{m}&\widehat{F}_{m}& &&\widehat{A}_{1}& & \\
&\widehat{C}_{m+1}&\widehat{F}_{m+1}& &&\widehat{A}_{2}& \\
& &\ddots  & & &&\ddots&\\
& &  &\widehat{C}_n&&&&\widehat{A}_{n}\\
\end{smallmatrix}
\end{pmatrix}
+r\begin{pmatrix}
\begin{smallmatrix}
\widehat{G}_{m}&\widehat{D}_{m+1}& \\
&\widehat{G}_{m+1}&\widehat{D}_{m+2}\\
& &\ddots \\
& &\ddots &\widehat{D}_n\\
& & &\widehat{G}_n\\
\widehat{B}_{1}& & \\
&\widehat{B}_{2}& \\
& &\ddots\\
& &&\widehat{B}_n
\end{smallmatrix}
\end{pmatrix},
\end{align}
\begin{align}\label{july25equ022}
&r\begin{pmatrix}
\begin{smallmatrix}
\widehat{D}_{m}& & & & & & & & \\
\widehat{E}_{m}&\widehat{F}_{m}& & & & &&\widehat{A}_{1}& & \\
\widehat{G}_{m}& &\widehat{D}_{m+1}& & & & & & \\
&\widehat{C}_{m+1}&-\widehat{E}_{m+1}&\widehat{F}_{m+1}& & & &&\widehat{A}_{2}& \\
& &\widehat{G}_{m+1}& &\widehat{D}_{3}& & & &&\ddots \\
& &&\ddots &\ddots &\ddots & && & \\
& & & &\widehat{C}_{n}&(-1)^{n-m}\widehat{E}_{n}&\widehat{F}_n& &&&\widehat{A}_{3}\\
\widehat{B}_{m}& & & & & & & & \\
& &\widehat{B}_{m+1}& & & & & & \\
& &  &\ddots && & & & \\
& & & & &\widehat{B}_n& & & 
\end{smallmatrix}
\end{pmatrix}\nonumber\\
&=r\begin{pmatrix}
\begin{smallmatrix}
\widehat{F}_{m}& & &&\widehat{A}_{1}& & \\
\widehat{C}_{m+1}&\widehat{F}_{m+1}& & &&\widehat{A}_{2}& \\
&\ddots &\ddots & & &&\ddots \\
& &\widehat{C}_{n}&\widehat{F}_n& &&&\widehat{A}_{n}
\end{smallmatrix}
\end{pmatrix}
+r\begin{pmatrix}
\begin{smallmatrix}
\widehat{D}_{m}& & \\
\widehat{G}_{m}&\widehat{D}_{m+1}& \\
&\widehat{G}_{m+1}&\widehat{D}_{m+2}\\
&&\ddots&\ddots \\
&&&\widehat{G}_{n}&\widehat{D}_{n}\\
\widehat{B}_{m}& & \\
&\widehat{B}_{m+1}& \\
&&\ddots \\
& &&&\widehat{B}_{n}
\end{smallmatrix}
\end{pmatrix}.
\end{align}
Now we  prove that (\ref{july25equ017})--(\ref{july25equ022}) is equalivent to (\ref{july25equ002})--(\ref{july25equ007}). To do this it is useful to establish the following facts.
\begin{description}
  \item[Fact 1] The expressions of $\widehat{E}_{i}$  in (\ref{july25equ016}): Note that
\begin{align*}
Z_{i+1}^{1}:=M_{ii}^{\dag}E_{ii}D_{ii}^{\dag}+S_{ii}^{\dag}S_{ii}C_{ii}^{\dag}E_{ii}N_{ii}^{\dag}
\end{align*}
and
\begin{align*}
Z_{i+1}^{2}:=&A_{i+1,i+1}^{\dag}E_{i+1,i+1}B_{i+1,i+1}^{\dag}-A_{i+1,i+1}^{\dag}C_{i+1,i+1}M_{i+1,i+1}^{\dag}E_{i+1,i+1}B_{i+1,i+1}^{\dag}\\
&-A_{i+1,i+1}^{\dag}S_{i+1,i+1}C_{i+1,i+1}^{\dag
}E_{i+1,i+1}N_{i+1,i+1}^{\dag}D_{i+1,i+1}B_{i+1,i+1}^{\dag}
\end{align*}are special solutions to the equations (\ref{july25equ009}) ($k=i$ and $k=i+1$), respectively. Hence,
\begin{align}\label{april15equ344}
\widehat{E}_{i}=Z_{i+1}^{2}-Z_{i+1}^{1}.
\end{align}
  \item[Fact 2] Formulas about $S_{ii}$: From
  \begin{align}
  S_{ii}-A_{ii}A_{ii}^{\dag}S_{ii}=R_{A_{ii}}S_{ii}=R_{A_{ii}}C_{ii}L_{M_{ii}}=M_{ii}L_{M_{ii}}=0,
  \end{align}we infer that
  \begin{align}\label{april15equ347}
  A_{ii}A_{ii}^{\dag}S_{ii}=S_{ii}.
  \end{align}
  \item[Fact 3] The ranks of $r\begin{pmatrix}R_{D_{ii}}\\R_{N_{ii}}\end{pmatrix}-r(R_{N_{ii}})$: 
  Applying Lemma \ref{hlemma03} to $r\begin{pmatrix}R_{D_{ii}}\\R_{N_{ii}}\end{pmatrix}-r(R_{N_{ii}})$ gives
  \begin{align*}
  &r\begin{pmatrix}R_{D_{ii}}\\R_{N_{ii}}\end{pmatrix}-r(R_{N_{ii}})\\=&
  r\begin{pmatrix}I&D_{ii}&0\\I&0&N_{ii}\end{pmatrix}-r\begin{pmatrix}I&N_{ii}\end{pmatrix}-r(D_{ii})\\=&
  r\begin{pmatrix}D_{ii}&N_{ii}\end{pmatrix}-r(D_{ii})\xlongequal{N_{ii}=D_{ii}L_{B_{ii}}}~0.
  \end{align*}Hence, we have
  \begin{align}
  r\begin{pmatrix}R_{D_{ii}}\\R_{N_{ii}}\end{pmatrix}=r(R_{N_{ii}}),
  \end{align}i.e.,
  \begin{align}\label{april15equ350}
  R_{D_{ii}}=U_{i}R_{N_{ii}},
  \end{align}where $U_{i}$ is a matrix.
\item[Fact 4] Formulas about $R_{N_{i+1,i+1}}D_{i+1,i+1}B_{i+1,i+1}^{\dag}B_{i+1,i+1}$: Note that
\begin{align*}
&R_{N_{i+1,i+1}}D_{i+1,i+1}-R_{N_{i+1,i+1}}D_{i+1,i+1}B_{i+1,i+1}^{\dag}B_{i+1,i+1}\\&=R_{N_{i+1,i+1}}D_{i+1,i+1}L_{B_{i+1,i+1}}=R_{N_{i+1,i+1}}N_{i+1,i+1}=0,
\end{align*}
Hence, we have
\begin{align}
R_{N_{i+1,i+1}}D_{i+1,i+1}B_{i+1,i+1}^{\dag}B_{i+1,i+1}=R_{N_{i+1,i+1}}D_{i+1,i+1}.\label{april15equ352}
\end{align}
\end{description}

Firstly, we prove that   (\ref{july25equ004})--(\ref{july25equ007})$\Longleftrightarrow$  (\ref{july25equ019})--(\ref{july25equ022}) for the case $n-m=1$, respectively.  It follows from Lemma \ref{hlemma03} that
\begin{align}
 r\begin{pmatrix}\widehat{E}_{i}&\widehat{C}_{i}&\widehat{F}_{i}&\widehat{A}_{i}\\\widehat{B}_{i}&0&0&0\end{pmatrix}=
r\begin{pmatrix}\widehat{A}_{i}&\widehat{C}_{i}&\widehat{F}_{i}\end{pmatrix}+r(\widehat{B}_{i})
\end{align}
\begin{align*}
\Leftrightarrow&
 r\begin{pmatrix}
 \begin{smallmatrix}
 \widehat{E}_{i}&L_{M_{ii}}L_{S_{ii}}&L_{A_{i+1,i+1}}&L_{M_{ii}}&A_{i+1,i+1}^{\dag}S_{i+1,i+1}\\
 R_{D_{ii}}&0&0&0&0\\
 R_{B_{i+1,i+1}}&0&0&0&0
 \end{smallmatrix}
 \end{pmatrix}\\
 &=
 r\begin{pmatrix}
 \begin{smallmatrix}
 L_{M_{ii}}L_{S_{ii}}&L_{A_{i+1,i+1}}&L_{M_{ii}}&A_{i+1,i+1}^{\dag}S_{i+1,i+1}
 \end{smallmatrix}
 \end{pmatrix}+
 r\begin{pmatrix}
 \begin{smallmatrix}
 R_{D_{ii}}\\R_{B_{i+1,i+1}}
 \end{smallmatrix}
 \end{pmatrix}
\end{align*}
\begin{align*}
\Leftrightarrow
 r\begin{pmatrix}
\begin{smallmatrix}
\widehat{E}_{i}&L_{A_{i+1,i+1}}&L_{M_{ii}}&A_{i+1,i+1}^{\dag}S_{i+1,i+1}\\
R_{D_{ii}}&0&0&0\\
R_{B_{i+1,i+1}}&0&0&0
\end{smallmatrix}
\end{pmatrix}=
r\begin{pmatrix}
\begin{smallmatrix}
L_{A_{i+1,i+1}}&L_{M_{ii}}&A_{i+1,i+1}^{\dag}S_{i+1,i+1}
\end{smallmatrix}
\end{pmatrix}+
r\begin{pmatrix}
\begin{smallmatrix}
R_{D_{ii}}\\R_{B_{i+1,i+1}}
 \end{smallmatrix}
 \end{pmatrix}
\end{align*}
\begin{align*}
\Leftrightarrow
 r\begin{pmatrix}
 \begin{smallmatrix}
 \widehat{E}_{i}&I&I&A_{i+1,i+1}^{\dag}S_{i+1,i+1}&0&0\\
 I&0&0&0&D_{ii}&0\\
 I&0&0&0&0&B_{i+1,i+1}\\
 0&0&M_{ii}&0&0&0\\
 0&A_{i+1,i+1}&0&0&0&0
 \end{smallmatrix}
 \end{pmatrix}=
 r\begin{pmatrix}
 \begin{smallmatrix}
 I&I&A_{i+1,i+1}^{\dag}S_{i+1,i+1}\\
 0&M_{ii}&0\\
 A_{i+1,i+1}&0&0
 \end{smallmatrix}
 \end{pmatrix}+r\begin{pmatrix}
 \begin{smallmatrix}
 I&D_{ii}&0\\
 I&0&B_{i+1,i+1}
 \end{smallmatrix}
 \end{pmatrix}
\end{align*}
\begin{align*}
\xLeftrightarrow{\mbox{By using}~ (\ref{april15equ347})}
 r\begin{pmatrix}
\begin{smallmatrix}
\widehat{E}_{i}&I&I&0&0&0\\
I&0&0&0&D_{ii}&0\\
I&0&0&0&0&B_{i+i,i+i}\\
0&0&M_{ii}&0&0&0\\
0&A_{i+i,i+i}&0&S_{i+i,i+i}&0&0
\end{smallmatrix}
\end{pmatrix}=
r\begin{pmatrix}
\begin{smallmatrix}
I&I&0\\0&M_{ii}&0\\A_{i+i,i+i}&0&S_{i+i,i+i}
\end{smallmatrix}
\end{pmatrix}+r\begin{pmatrix}
\begin{smallmatrix}
I&D_{ii}&0\\
I&0&B_{i+i,i+i}
 \end{smallmatrix}
 \end{pmatrix}
\end{align*}
\begin{align*}
\Leftrightarrow
 r\begin{pmatrix}
 \begin{smallmatrix}
 \widehat{E}_{i}&I&I&0&0&0\\
 I&0&0&0&D_{ii}&0\\
 I&0&0&0&0&B_{i+1,i+1}\\
 0&0&M_{ii}&0&0&0\\
 0&A_{i+1,i+1}&0&C_{i+1,i+1}&0&0
 \end{smallmatrix}
 \end{pmatrix}=
 r\begin{pmatrix}
 \begin{smallmatrix}
 I&I&0\\0&M_{ii}&0\\A_{i+1,i+1}&0&C_{i+1,i+1}
 \end{smallmatrix}
 \end{pmatrix}+r\begin{pmatrix}
 \begin{smallmatrix}
 I&D_{ii}&0\\
 I&0&B_{i+1,i+1}
 \end{smallmatrix}
 \end{pmatrix}
\end{align*}
\begin{align*}
\Leftrightarrow
 r\begin{pmatrix}
 \begin{smallmatrix}
 \widehat{E}_{i}&I&I&0&0&0&0\\
 I&0&0&0&D_{ii}&0&0\\
 I&0&0&0&0&B_{i+i,i+i}&0\\
 0&0&C_{ii}&0&0&0&A_{ii}\\
 0&A_{i+i,i+i}&0&C_{i+i,i+i}&0&0&0
 \end{smallmatrix}
 \end{pmatrix}=
 r\begin{pmatrix}
 \begin{smallmatrix}
 I&I&0&0\\0&C_{ii}&0&A_{ii}\\A_{i+i,i+i}&0&C_{i+i,i+i}&0
 \end{smallmatrix}
 \end{pmatrix}+r\begin{pmatrix}
 \begin{smallmatrix}
 I&D_{ii}&0\\
 I&0&B_{i+i,i+i}
 \end{smallmatrix}
 \end{pmatrix}
\end{align*}
\begin{align*}
\Leftrightarrow
 r\begin{pmatrix}
 \begin{smallmatrix}
 \widehat{E}_{i}&I&I&0&0&0&0&0&0\\
 I&0&0&0&G_{i}&0&0&0&0\\
 I&0&0&0&0&D_{i+1}&0&0&0\\
 0&0&F_{i}&0&0&0&C_{i}&A_{i}&0\\
 0&C_{i+1}&0&F_{i+1}&0&0&0&0&A_{i+1}\\
 0&0&0&0&B_{i}&0&0&0&0\\
 0&0&0&0&0&B_{i+1}&0&0&0
 \end{smallmatrix}
 \end{pmatrix}=
 r\begin{pmatrix}
 \begin{smallmatrix}
 I&I&0&0&0&0\\0&F_{i}&0&C_{i}&A_{i}&0\\C_{i+1}&0&F_{i+1}&0&0&A_{i+1}
 \end{smallmatrix}
 \end{pmatrix}
 +r\begin{pmatrix}
 \begin{smallmatrix}
 I&G_{i}&0\\
 I&0&D_{i+1}\\
 0&B_{i}&0\\
 0&0&B_{i+1}
 \end{smallmatrix}
 \end{pmatrix}
\end{align*}
\begin{align*}
\xLeftrightarrow{\mbox{By using}~ (\ref{april15equ344})}
 r\begin{pmatrix}
\begin{smallmatrix}
Z_{i+1}^{2}-Z_{i+1}^{1}&I&I&0&0&0&0&0&0\\
I&0&0&0&G_{i}&0&0&0&0\\
I&0&0&0&0&D_{i+1}&0&0&0\\
0&0&F_{i}&0&0&0&C_{i}&A_{i}&0\\
0&C_{i+1}&0&F_{i+1}&0&0&0&0&A_{i+1}\\
0&0&0&0&B_{i}&0&0&0&0\\
0&0&0&0&0&B_{i+1}&0&0&0
\end{smallmatrix}
\end{pmatrix}=
r\begin{pmatrix}
\begin{smallmatrix}
I&I&0&0&0&0\\0&F_{i}&0&C_{i}&A_{i}&0\\C_{i+1}&0&F_{i+1}&0&0&A_{i+1}
\end{smallmatrix}
\end{pmatrix}
+r\begin{pmatrix}
\begin{smallmatrix}
I&G_{i}&0\\
I&0&D_{i+1}\\
0&B_{i}&0\\
0&0&B_{i+1}
 \end{smallmatrix}
 \end{pmatrix}
\end{align*}
\begin{align*}
\Leftrightarrow r\begin{pmatrix}
\begin{smallmatrix}
0&0&B_{i}&0&0&0&0\\
A_{i}&C_{i}&E_{i}&F_{i}&0&0&0\\
0&0&G_{i}&0&D_{i+1}&0&0\\
0&0&0&C_{i+1}&-E_{i+1}&F_{i+1}&A_{i+1}\\
0&0&0&0&B_{i+1}&0&0
\end{smallmatrix}
\end{pmatrix}=r\begin{pmatrix}
\begin{smallmatrix}
A_{i}&C_{i}&F_{i}&0&0\\0&0&C_{i+1}&F_{i+1}&A_{i+1}
\end{smallmatrix}
\end{pmatrix}
+r\begin{pmatrix}
\begin{smallmatrix}
B_{i}&0\\G_{i}&D_{i+1}\\0&B_{i+1}
\end{smallmatrix}
\end{pmatrix}\Leftrightarrow(\ref{july25equ019}).
\end{align*}
Similarly, we have that
\begin{align*}
(\ref{july25equ005})\xLeftrightarrow{\mbox{By using}~ (\ref{april15equ344}),~(\ref{april15equ350}),~(\ref{april15equ352}) ~\mbox{and~elementary~ operations}} (\ref{july25equ020}),
\end{align*}
\begin{align*}
(\ref{july25equ006}) \xLeftrightarrow{\mbox{By using}~ (\ref{april15equ344}),~(\ref{april15equ352}) ~\mbox{and~elementary~ operations}}(\ref{july25equ021}),
\end{align*}
\begin{align*}
(\ref{july25equ007}) \xLeftrightarrow{\mbox{By using}~ (\ref{april15equ344}),~(\ref{april15equ347}),  ~(\ref{april15equ350})~\mbox{and~elementary~ operations}} (\ref{july25equ022}).
\end{align*}

Secondly,  we prove that   (\ref{july25equ004})--(\ref{july25equ007})$\Longleftrightarrow$  (\ref{july25equ019})--(\ref{july25equ022}) for the case $n-m>1$, respectively. 
Replace the notations in (\ref{july25equ015})--(\ref{july25equ017}) into (\ref{july25equ020}), we obtain:

\begin{align*}
&r\begin{pmatrix}
\begin{smallmatrix}
\widehat{C}_{m}&\widehat{E}_{m}&\widehat{F}_{m}& & & & &\widehat{A}_{m}& & & \\
&\widehat{G}_{m}& &\widehat{D}_{m+1}& & & & & & & \\
& &\widehat{C}_{m+1}&-\widehat{E}_{m+1}&\widehat{F}_{m+1}& & & &\widehat{A}_{m+1}& & \\
& & &\ddots&\ddots&\ddots & & & &\ddots& \\
& & & &\widehat{C}_{n}&(-1)^{n-m}\widehat{E}_{n}&\widehat{F}_{n}& & & &\widehat{A}_{n}\\
&\widehat{B}_{m}& & & & & & & & & \\
& & &\widehat{B}_{m+1}& & & & & & & \\
& & & &\ddots & & & & & & \\
& & & & &\widehat{B}_{n}& & & & & 
\end{smallmatrix}
\end{pmatrix}\nonumber\\
&=r\begin{pmatrix}
\begin{smallmatrix}
\widehat{C}_{m}&\widehat{F}_{m}& & &\widehat{A}_{m}& & & \\
&\widehat{C}_{m+1}&\widehat{F}_{m+1}& & &\widehat{A}_{m+1}& & \\
& &\ddots&\ddots & & &\ddots & \\
& & &\widehat{C}_n &\widehat{F}_n& & &\widehat{A}_n
\end{smallmatrix}
\end{pmatrix}
+r\begin{pmatrix}
\begin{smallmatrix}
\widehat{G}_{m}&\widehat{D}_{m+1}& & & \\
&\widehat{G}_{m+1}&\widehat{D}_{m+2}& & \\
& &\ddots&\ddots& \\
& & &\widehat{G}_{n-2}&\widehat{D}_{n}\\
\widehat{B}_{m}& & & & \\
&\widehat{B}_{m+1}& & & \\
& & &\ddots& \\
& & & &\widehat{B}_{n}
\end{smallmatrix}
\end{pmatrix},
\end{align*}
\begin{align*}
\xLeftrightarrow{\mbox{Replace the notations in (\ref{july25equ015})--(\ref{july25equ017})}}
\end{align*}

\begin{align*}
&r\begin{pmatrix}
\begin{smallmatrix}
L_{M_{m,m}}&\widehat{E}_m&A_{m+1,m+1}^{\dag}S_{m+1,m+1}& & & & &\begin{pmatrix}\begin{smallmatrix}L_{M_{mm}}L_{S_{mm}}&-L_{A_{m+1,m+1}}\end{smallmatrix}\end{pmatrix}& & & \\
&R_{N_{m+1,m+1}}D_{m+1,m+1}B_{m+1,m+1}^{\dag}& &\widehat{D}_{m+1}& & & & & & & \\
& &\widehat{C}_{m+1}&-\widehat{E}_{m+1}&\widehat{F}_{m+1}& & & &\widehat{A}_{m+1}& & \\
& & &\ddots&\ddots&\ddots & & & &\ddots& \\
&\begin{pmatrix}\begin{smallmatrix}R_{D_{mm}}\\-R_{B_{m+1,m+1}}\end{smallmatrix}\end{pmatrix}& & & & & & & & & \\
& & &\widehat{B}_{m+1}& & & & & & & \\
& & & &\ddots & & & & & & 
\end{smallmatrix}
\end{pmatrix}\nonumber\\
&=r\begin{pmatrix}
\begin{smallmatrix}
L_{M_{m,m}}&A_{m+1,m+1}^{\dag}S_{m+1,m+1}& & &\begin{pmatrix}\begin{smallmatrix}L_{M_{mm}}L_{S_{mm}}&-L_{A_{m+1,m+1}}\end{smallmatrix}\end{pmatrix}& & & \\
&\widehat{C}_{m+1}&\widehat{F}_{m+1}& & &\widehat{A}_{m+1}& & \\
& &\ddots&\ddots & & &\ddots & 
\end{smallmatrix}
\end{pmatrix}\\
&+r\begin{pmatrix}
\begin{smallmatrix}
R_{N_{m+1,m+1}}D_{m+1,m+1}B_{m+1,m+1}^{\dag}&\widehat{D}_{m+1}& & & \\
&\widehat{G}_{m+1}&\widehat{D}_{m+2}& & \\
& &\ddots&\ddots& \\
\begin{pmatrix}\begin{smallmatrix}R_{D_{mm}}\\-R_{B_{m+1,m+1}}\end{smallmatrix}\end{pmatrix}& & & & \\
&\widehat{B}_{m+1}& & & \\
& & &\ddots& 
\end{smallmatrix}
\end{pmatrix},
\end{align*}
\begin{align*}
&\Leftrightarrow r\begin{pmatrix}
\begin{smallmatrix}
L_{M_{m,m}}&\widehat{E}_m&A_{m+1,m+1}^{\dag}S_{m+1,m+1}& & & & &-L_{A_{m+1,m+1}}& & & \\
&R_{N_{m+1,m+1}}D_{m+1,m+1}B_{m+1,m+1}^{\dag}& &R_{N_{m+1,m+1}}& & & & & & & \\
& &\widehat{C}_{m+1}&-\widehat{E}_{m+1}&\widehat{F}_{m+1}& & & &\widehat{A}_{m+1}& & \\
& & &\ddots&\ddots&\ddots & & & &\ddots& \\
&\begin{pmatrix}\begin{smallmatrix}R_{D_{mm}}\\-R_{B_{m+1,m+1}}\end{smallmatrix}\end{pmatrix}& & & & & & & & & \\
& & &\widehat{B}_{m+1}& & & & & & & \\
& & & &\ddots & & & & & & 
\end{smallmatrix}
\end{pmatrix}\nonumber\\
&=r\begin{pmatrix}
\begin{smallmatrix}
L_{M_{m,m}}&A_{m+1,m+1}^{\dag}S_{m+1,m+1}& & &-L_{A_{m+1,m+1}}& & & \\
&\widehat{C}_{m+1}&\widehat{F}_{m+1}& & &\widehat{A}_{m+1}& & \\
& &\ddots&\ddots & & &\ddots & 
\end{smallmatrix}
\end{pmatrix}\\
&+r\begin{pmatrix}
\begin{smallmatrix}
R_{N_{m+1,m+1}}D_{m+1,m+1}B_{m+1,m+1}^{\dag}&R_{N_{m+1,m+1}}& & & \\
&\widehat{G}_{m+1}&\widehat{D}_{m+2}& & \\
& &\ddots&\ddots& \\
\begin{pmatrix}\begin{smallmatrix}R_{D_{mm}}\\-R_{B_{m+1,m+1}}\end{smallmatrix}\end{pmatrix}& & & & \\
&\widehat{B}_{m+1}& & & \\
& & &\ddots& 
\end{smallmatrix}
\end{pmatrix}
\end{align*}

\begin{align*}
\Leftrightarrow &r\begin{pmatrix}
\begin{smallmatrix}
 & &R_{D_{mm}}& & & & \\
& &R_{B_{m+1,m+1}}& & & & \\
L_{A_{m+1,m+1}}&L_{M_{mm}}&\widehat{E}_{m}&A_{m+1,m+1}^{\dag}S_{m+1,m+1}& & & \\
& &R_{N_{m+1,m+1}}D_{m+1,m+1}B_{m+1,m+1}^{\dag}& &R_{N_{m+1,m+1}}& & \\
& & &L_{M_{m+1,m+1}}&-\widehat{E}_{m+1}&A_{m+2,m+2}^{\dag}S_{m+2,m+2}&L_{A_{m+2,m+2}}\\
& & & &R_{D_{m+1,m+1}}& & \\
& & & &R_{B_{m+2,m+2}}& & \\
& & & &\ddots&\ddots &\ddots 
\end{smallmatrix}
\end{pmatrix}\\&=
r\begin{pmatrix}
\begin{smallmatrix}
L_{A_{m+1,m+1}}&L_{M_{mm}}&A_{m+1,m+1}^{\dag}S_{m+1,m+1}& & \\ & &L_{M_{m+1,m+1}}&A_{m+2,m+2}^{\dag}S_{m+2,m+2}&L_{A_{m+2,m+2}}\\
& & & &\ddots &\ddots 
\end{smallmatrix}
\end{pmatrix}\\
&+r\begin{pmatrix}
\begin{smallmatrix}
R_{D_{mm}}& \\ R_{B_{m+1,m+1}}& \\ R_{N_{m+1,m+1}}D_{m+1,m+1}B_{m+1,m+1}^{\dag}&R_{N_{m+1,m+1}}\\ &R_{D_{m+1,m+1}}\\  &R_{B_{m+2,m+2}}\\
&&\ddots &\ddots 
\end{smallmatrix}
\end{pmatrix}
\end{align*}

\begin{align*}
\Leftrightarrow &r\begin{pmatrix}
\begin{smallmatrix}
 &R_{D_{mm}}& & & && \\
 &R_{B_{m+1,m+1}}& & & && \\
L_{M_{mm}}&\widehat{E}_{m}&A_{m+1,m+1}^{\dag}S_{m+1,m+1}& & &I \\
 &R_{N_{m+1,m+1}}D_{m+1,m+1}B_{m+1,m+1}^{\dag}& &R_{N_{m+1,m+1}}& && \\
 & &L_{M_{m+1,m+1}}&-\widehat{E}_{m+1}&A_{m+2,m+2}^{\dag}S_{m+2,m+2}&&I\\
 & & &R_{D_{m+1,m+1}}& \ddots &&&\ddots  \\
 & & &R_{B_{m+2,m+2}}&\ddots  & \\
 & & & &\ddots & A_{m+1,m+1}& \\
 & & & & &&A_{m+2,m+2}\\
  & & & & &&&\ddots 
\end{smallmatrix}
\end{pmatrix}\\&=
r\begin{pmatrix}
\begin{smallmatrix}
&L_{M_{mm}}&A_{m+1,m+1}^{\dag}S_{m+1,m+1}&& I&& \\
 & &L_{M_{m+1,m+1}}&A_{m+2,m+2}^{\dag}S_{m+2,m+2}&&I\\
& & \ddots &\ddots &  &&\ddots  \\ & & & &&\\
& &  && A_{m+1,m+1} & \\ & & & &&A_{m+2,m+2}\\
& & & &&&\ddots 
\end{smallmatrix}
\end{pmatrix}\\
&+r\begin{pmatrix}\begin{smallmatrix}R_{D_{mm}}& \\ R_{B_{m+1,m+1}}& \\ R_{N_{m+1,m+1}}D_{m+1,m+1}B_{m+1,m+1}^{\dag}&R_{N_{m+1,m+1}}\\ &R_{D_{m+1,m+1}}\\  &R_{B_{m+2,m+2}}\\
&&\ddots&\ddots 
\end{smallmatrix}\end{pmatrix}
\end{align*}

\begin{align*}
\xLeftrightarrow{\mbox{By}~ (\ref{april15equ347})} &r\begin{pmatrix}
\begin{smallmatrix}
& &R_{D_{ mm}}& & & & \\
& &R_{B_{ m+1,m+1}}& & & & \\
&L_{M_{ mm}}&\widehat{E}_{1}& & && I& \\
& &R_{N_{ m+1,m+1}}D_{ m+1,m+1}B_{ m+1,m+1}^{\dag}& &R_{N_{ m+1,m+1}}& & \\
& & &L_{M_{ m+1,m+1}}&-\widehat{E}_{2}& &&I\\
& & & &R_{D_{ m+1,m+1}}& &&&\ddots  \\
& & & &R_{B_{ m+2,m+2}}& & \\
& & &S_{ m+1,m+1}& && A_{ m+1,m+1}& \\
& & & & \ddots&S_{ m+2,m+2}&&A_{ m+2,m+2}\\
& & & & &&\ddots &\ddots&\ddots  \\
\end{smallmatrix}
\end{pmatrix}\\&=
r\begin{pmatrix}
\begin{smallmatrix}
&L_{M_{ mm}}& & &I& \\ 
& &L_{M_{ m+1,m+1}}& &&I\\
& &S_{ m+1,m+1}&& A_{ m+1,m+1}&&\ddots  \\ 
& & &S_{ m+2,m+2}&&A_{ m+2,m+2}\\
&&&&\ddots &\ddots &\ddots 
\end{smallmatrix}
\end{pmatrix}\\
&+r\begin{pmatrix}
\begin{smallmatrix}
R_{D_{ mm}}& \\ R_{B_{ m+1,m+1}}& \\ R_{N_{ m+1,m+1}}D_{ m+1,m+1}B_{ m+1,m+1}^{\dag}&R_{N_{ m+1,m+1}}\\ &R_{D_{ m+1,m+1}}\\  &R_{B_{ m+2,m+2}}\\
&\ddots &\ddots 
\end{smallmatrix}
\end{pmatrix}
\end{align*}
\begin{align*}
\Leftrightarrow &r\begin{pmatrix}
\begin{smallmatrix}
& &R_{D_{mm}}& & & & & \\
&L_{M_{mm}}&\widehat{E}_{m}& & && I& & \\
& &R_{N_{m+1,m+1}}D_{m+1,m+1}B_{m+1,m+1}^{\dag}& &R_{N_{m+1,m+1}}& & & \\
& & &I&-\widehat{E}_{m+1}& &&I& \\
& & & &R_{D_{m+1,m+1}}& & &&\ddots  \\
& & & &R_{B_{m+2,m+2}}& & & \\
& & &C_{m+1,m+1}& & &A_{m+1,m+1}& & \\
& & & & &C_{m+2,m+2}&\ddots &A_{m+2,m+2}& \\
& &I& & & & \ddots &\ddots &B_{m+1,m+1}\\
\end{smallmatrix}
\end{pmatrix}\\&=
r\begin{pmatrix}
\begin{smallmatrix}
&L_{M_{mm}}& && I& \\
& &I& &&&I\\
& &C_{m+1,m+1}& &A_{m+1,m+1}&& &\ddots \\ 
& & &C_{m+2,m+2}&\ddots& &A_{m+2,m+2}\\
&&&&\ddots &&&\ddots 
\end{smallmatrix}
\end{pmatrix}\\
&+r\begin{pmatrix}
\begin{smallmatrix}
R_{D_{mm}}& & \\
 R_{N_{m+1,m+1}}D_{m+1,m+1}B_{m+1,m+1}^{\dag}&R_{N_{m+1,m+1}}& \\ &R_{D_{m+1,m+1}}& \\  &R_{B_{m+2,m+2}}& \\
  I& \ddots& \ddots &B_{m+1,m+1}\\ 
  & & & &\ddots 
\end{smallmatrix}
\end{pmatrix}
\end{align*}

\begin{align*}
\xLeftrightarrow{\mbox{By}~ (\ref{april15equ352})} &
r\begin{pmatrix}
\begin{smallmatrix}
& &R_{D_{mm}}& & & & & \\
&L_{M_{mm}}&\widehat{E}_{m}& & &&I & & \\
& & & &R_{N_{m+1,m+1}}& && &R_{N_{m+1,m+1}}D_{m+1,m+1}\\
& & &I&-\widehat{E}_{m+1}& &&I& \\
& & & &R_{D_{m+1,m+1}}& & & \\
& & & &R_{B_{m+2,m+2}}& & & \\
& & &C_{m+1,m+1}& && A_{m+1,m+1}& &\ddots  \\
& & & & &C_{m+2,m+2}&\ddots &A_{m+2,m+2}& \\
& &I& & && \ddots &\ddots  &B_{m+1,m+1}\\
\end{smallmatrix}
\end{pmatrix}=\\&
r\begin{pmatrix}
\begin{smallmatrix}
&L_{M_{mm}}& && I& \\ 
& &I& &&&I\\
& &C_{m+1,m+1}&& A_{m+1,m+1}&&&\ddots  \\
 & & &C_{m+2,m+2}&&&A_{m+2,m+2}\\
 &&&&\ddots &\ddots &&\ddots 
\end{smallmatrix}
\end{pmatrix}\\
&+r\begin{pmatrix}
\begin{smallmatrix}
R_{D_{mm}}& & \\ 
 &R_{N_{m+1,m+1}}&R_{N_{m+1,m+1}}D_{m+1,m+1}\\ 
 &R_{D_{m+1,m+1}}& \\ 
  &R_{B_{m+2,m+2}}& \\
  I& \ddots&B_{m+1,m+1}\\ 
  & & & &\ddots 
\end{smallmatrix}
\end{pmatrix}
\end{align*}

\begin{align*}
\Leftrightarrow &r\begin{pmatrix}
\begin{smallmatrix}
&  &R_{D_{mm}}&  &  &  &  &  \\
&L_{M_{mm}}&\widehat{E}_{m}&  &  && I &  &  \\
&  &  &  &I&  &  &&&D_{m+1,m+1}\\
&  &  &I&-\widehat{E}_{m+1}&  &&I&  \\
&  &  &  &R_{B_{m+2,m+2}}&  &  &  \\
&  &  &C_{m+1,m+1}&  &&  A_{m+1,m+1}&  &  \\
&  &  &  & \ddots &C_{m+2,m+2}&&A_{m+2,m+2}&  \\
&  &&  &  &  \ddots &\ddots   &\ddots & &\\
&  &I&  &  &   &   && &B_{m+1,m+1}\\
\end{smallmatrix}
\end{pmatrix}\\&=
r\begin{pmatrix}
\begin{smallmatrix}
&L_{M_{mm}}&  && I &  \\  
&  &I&  &&I\\
&  &C_{m+1,m+1}&  &A_{m+1,m+1}& &\ddots  \\
  &  &  &C_{m+2,m+2}&&A_{m+2,m+2}\\
  &&&&\ddots &&\ddots 
\end{smallmatrix}
\end{pmatrix}\\
&+r\begin{pmatrix}
\begin{smallmatrix}
R_{D_{mm}}&  &  \\ 
&I&&&D_{m+1,m+1}\\ 
&R_{B_{m+2,m+2}}&  \\
&  \ddots & &&\\ 
I&   & &&B_{m+1,m+1}\\ 
& & & &\ddots 
\end{smallmatrix}
\end{pmatrix}
\end{align*}
\begin{align*}
\Leftrightarrow &r\begin{pmatrix}
\begin{smallmatrix}
&  &I&  &  &  &  &  &D_{mm}&  &  \\
&I&\widehat{E}_{m}&  &  &  &  &  &  &  &&I  \\
&  &  &  &I&  &  &D_{m+1,m+1}&  &  &  \\
&  &  &I&-\widehat{E}_{m+1}&  &&  &  &  & &&I \\
&  &  &  &I&  &  &  &  &B_{m+2,m+2}&  \\
&  &  &C_{m+1,m+1}&  &  &  &  &  &  && A_{m+1,m+1}&&&\ddots  \\
&  &  &  &  &C_{m+2,m+2}&&  &  &  &&&A_{m+2,m+2}  \\
&C_{mm}&  &  &  &  &  & \ddots  &  &  &A_{mm}&&&\ddots \\
&&&&&&&&&\ddots &&&\ddots \\
&  &I&  &  &  &  &B_{m+1,m+1}&  &  &  \\
\end{smallmatrix}
\end{pmatrix}=\\&
r\begin{pmatrix}
\begin{smallmatrix}
&I&  &  &  &  I\\
&  &I&  & & & I  \\
&  &C_{m+1,m+1}&  &  &A_{m+1,m+1}  \\
&  &  &C_{m+2,m+2}&&&A_{m+2,m+2}&  \\
&C_{mm}&  &  &  &A_{mm}&\ddots \\
& & & & \ddots &\ddots 
\end{smallmatrix}
\end{pmatrix}
+r\begin{pmatrix}
\begin{smallmatrix}
I&  &  &D_{mm}&  \\
&I&D_{m+1,m+1}&  &  \\
& &\ddots & \ddots &\\
I&  &B_{m+1,m+1}&  &  \\
&I&  &  B_{m+2,m+2}\\
& & & &\ddots 
\end{smallmatrix}
\end{pmatrix}
\end{align*}                      
\begin{align*}
\xLeftrightarrow{\mbox{By}~ (\ref{april15equ344})}&r\begin{pmatrix}
\begin{smallmatrix}
& &I& & & & & &G_{m}& & & & & \\
& &I& & & & &D_{m+1}& & & & & & \\
I&I&Z_{m+1}^{2}-Z_{m+1}^{1}& & & & & & & & & & & \\
& & & &I& & &G_{m+1}& & & & & & \\
& & &I&-Z_{m+2}^{2}+Z_{m+2}^{1}& &I& & & & & & & \\
& & & &I& & & & &D_{m+2}& & & & \\
C_{m+1}& & &F_{m+1}& & & & & & & &A_{m+1}& & \\
& & & & &F_{m+2}&C_{m+2}& & & & & &A_{m+2}& \\
&F_{m}& & & & & & & & &C_{m}& & &A_{m}\\
& & & & & & &B_{m+1}& & & & & &&&&\ddots  \\
& & & & & & & &B_{m}& & & & & \\
& & & & & & & & &B_{m+2}& & & & \\
& & & & & & & & &\ddots && \ddots & & & 
\end{smallmatrix}
\end{pmatrix}\\&=
r\begin{pmatrix}
\begin{smallmatrix}
I&I& & & & & & & \\ & &I& &I& & & & \\C_{m+1}& &F_{m+1}& & & &A_{m+1}& & \\ & & &F_{m+2}&C_{m+2}& & &A_{m+2}& \\ &F_{mm}& & & &C_{m}& & &A_{m}\\
& & & & &\ddots  &\ddots  & & & & \ddots & & 
\end{smallmatrix}
\end{pmatrix}
+r\begin{pmatrix}
\begin{smallmatrix}
I& & &G_{m}& \\ I& &D_{m+1}& & \\  &I&G_{m+1}& & \\  &I& & &D_{m+2}\\ & &B_{m+1}& & \\ & & &B_{m}& \\ & & & &B_{m+2}\\&&&\ddots&&\ddots 
\end{smallmatrix}
\end{pmatrix}
\end{align*}

\begin{align*}
&\Leftrightarrow r\begin{pmatrix}
\begin{smallmatrix}
& &I& & & & & &G_{ m}& & & & & \\
& &I& & & & &D_{m+1}& & & & & & \\
I&I& & & & & & & & & & & & \\
& & & &I& & &G_{m+1}& & & & & & \\
& & &I& & &I& & & & & & & \\
& & & &I& & & & &D_{m+2}& & & & \\
C_{m+1}& & &F_{m+1}& & & &E_{m+1}& & & &A_{m+1}& & \\
& & & & &F_{m+2}&C_{m+2}& & &-E_{m+2}& & &A_{m+2}& \\
&F_{ m}& & & & & & &-E_{ m}& &C_{ m}& & &A_{ m}\\
& & & & & & &B_{m+1}& & & & & &&\ddots  \\
& & & & & & & &B_{ m}& & & & & \\
& & & & & & & & &B_{m+2}& & & & \\
& & & & & & & & &\ddots && \ddots & & & 
\end{smallmatrix}
\end{pmatrix}=\\&
r\begin{pmatrix}
\begin{smallmatrix}
I&I& & & & & & & \\ & &I& &I& & & & \\C_{m+1}& &F_{m+1}& & & &A_{m+1}& & \\ & & &F_{m+2}&C_{m+2}& & &A_{m+2}& \\ &F_{ m m}& & & &C_{ m}& & &A_{ m}\\
& & & & &\ddots  &\ddots  & & & & \ddots & & 
\end{smallmatrix}
\end{pmatrix}
+r\begin{pmatrix}
\begin{smallmatrix}
I& & &G_{ m}& \\ I& &D_{m+1}& & \\  &I&G_{m+1}& & \\  &I& & &D_{m+2}\\ & &B_{m+1}& & \\ & & &B_{ m}& \\ & & & &B_{m+2}\\
&&&\ddots&&\ddots 
\end{smallmatrix}
\end{pmatrix}
\end{align*}
\begin{align*}
&\Leftrightarrow r\begin{pmatrix}
\begin{smallmatrix}
C_{m}&E_{m}&F_{m}& & & & &A_{m}& & & \\
&G_{m}& &D_{m+1}& & & & & & & \\
& &C_{m+1}&-E_{m+1}&F_{m+1}& & & &A_{m+1}& & \\
& & &\ddots&\ddots&\ddots & & & &\ddots& \\
&B_{m}& & & & & & & & & \\
& & &B_{m+1}& & & & & & & \\
& & & &\ddots & & & & & & \\
\end{smallmatrix}
\end{pmatrix}\nonumber\\
&=r\begin{pmatrix}
\begin{smallmatrix}
C_{m}&F_{m}& & &A_{m}& & & \\
&C_{m+1}&F_{m+1}& & &A_{m+1}& & \\
& &\ddots&\ddots & & &\ddots & \\
\end{smallmatrix}
\end{pmatrix}
+r\begin{pmatrix}
\begin{smallmatrix}
G_{m}&D_{m+1}& & & \\
&G_{m+1}&D_{m+2}& & \\
& &\ddots&\ddots& \\
& & &G_{n-2}&D_{n}\\
B_{m}& & & & \\
&B_{m+1}& & & \\
& & &\ddots& \\
\end{smallmatrix}
\end{pmatrix}.
\end{align*}
Following this way, we obtain that
\begin{align*}
&r\begin{pmatrix}
\begin{smallmatrix}
C_{m}&E_{m}&F_{m}& & & & &A_{m}& & & \\
&G_{m}& &D_{m+1}& & & & & & & \\
& &C_{m+1}&-E_{m+1}&F_{m+1}& & & &A_{m+1}& & \\
& & &\ddots&\ddots&\ddots & & & &\ddots& \\
& & & &C_{n}&(-1)^{n-m}E_{n}&F_{n}& & & &A_{n}\\
&B_{m}& & & & & & & & & \\
& & &B_{m+1}& & & & & & & \\
& & & &\ddots & & & & & & \\
& & & & &B_{n}& & & & & 
\end{smallmatrix}
\end{pmatrix}\nonumber\\
&=r\begin{pmatrix}
\begin{smallmatrix}
C_{m}&F_{m}& & &A_{m}& & & \\
&C_{m+1}&F_{m+1}& & &A_{m+1}& & \\
& &\ddots&\ddots & & &\ddots & \\
& & &C_n &F_n& & &A_n
\end{smallmatrix}
\end{pmatrix}
+r\begin{pmatrix}
\begin{smallmatrix}
G_{m}&D_{m+1}& & & \\
&G_{m+1}&D_{m+2}& & \\
& &\ddots&\ddots& \\
& & &G_{n-2}&D_{n}\\
B_{m}& & & & \\
&B_{m+1}& & & \\
& & &\ddots& \\
& & & &B_{n}
\end{smallmatrix}
\end{pmatrix}.
\end{align*}
Similarly, it can be found that :
\begin{align*}
(\ref{july25equ004}) \Leftrightarrow (\ref{july25equ020}),
~
(\ref{july25equ005}) \Leftrightarrow (\ref{july25equ021}),
~
(\ref{july25equ006}) \Leftrightarrow (\ref{july25equ022}),
\end{align*}
for case $n-m>1$.

\end{proof}

\section{\textbf{Some systems of quaternion matrix equations involving $\phi$-Hermicity}}
In this section, we will consider the system of quaternion matrix equations involving $\phi$-Hermicity
\begin{align}\label{june16equ051}
  \left\{\begin{array}{c}
A_{1}X_{1}+(A_{1}X_{1})_{\phi}+C_{1}Z_{1}(C_{1})_{\phi}+F_{1}Z_{2}(F_{1})_{\phi}=E_{1},\\
A_{2}X_{2}+(A_{2}X_{2})_{\phi}+C_{2}Z_{2}(C_{2})_{\phi}+F_{2}Z_{3}(F_{3})_{\phi}=E_{2},\\
\vdots \\
A_{k}X_{k}+(A_{k}X_{k})_{\phi}+C_{k}Z_{k}(C_{k})_{\phi}+F_{k}Z_{k+1}(F_{k})_{\phi}=E_{k},\\
Z_{1}=(Z_{1})_{\phi},~Z_{2}=(Z_{2})_{\phi},\cdots,Z_{k}=(Z_{k})_{\phi},~Z_{k+1}=(Z_{k+1})_{\phi},
\end{array}
  \right.
\end{align}
where $A_{i}, C_{i}, F_{i}, E_{i}$ are given, and $E_{i}$ are $\phi$-Hermitian matrices $(i=\overline{1,k})$. We derive some solvability conditions to the system (\ref{june16equ051}) in terms of ranks involved. We first give the definition of nonstandard involution and $\phi$-Hermitian.

A map $\phi$: $\mathbb{H}\longrightarrow \mathbb{H}$ is called an involution if $\phi(xy)=\phi(y)\phi(x)$,
$\phi(x+y)=\phi(x)+\phi(y)$ and $\phi(\phi(x))=x$ for all $x,y\in \mathbb{H}.$ Moreover, the matrix
representing of $\phi$, with respect to the basis $\{1,\mathbf{i},\mathbf{j},\mathbf{k}\}$ is
\begin{align}
\begin{pmatrix}1&0\\0&T\end{pmatrix},
\end{align}
where either $T=-I_{3}$ or $T$ is a $3\times 3$ real orthogonal symmetric matrix with eigenvalues $1,1,-1$ (see Theorem 2.4.4 in \cite{rodman}). If $T=-I_{3},$ then $\phi$ is the standard involution, and for the latter case, $\phi$ is called a nonstandard involution (see Definition 2.4.5 in \cite{rodman}). We in this paper consider only the nonstandard involution.

For $A\in\mathbb{H}^{m\times n},$ we denote by $A_{\phi}$ the $n\times m$ matrix obtained by applying $\phi$ entrywise to $A^{T},$ where $\phi$ is a nonstandard involution (\cite{rodman}, page 44). Now we recall the definition of the $\phi$-Hermitian matrix.

\begin{definition}[$\phi$-Hermitian Matrix]\label{def04}
\cite{rodman} $A\in\mathbb{H}^{n\times n}$ is said to be $\phi$-Hermitian if $A=A_{\phi}$, where $\phi$ is a nonstandard involution.
\end{definition}

Some properties of $A_{\phi}$, rank and Moore-Penrose inverse of $A_{\phi}$ are given as follows.

\begin{property}\cite{rodman}
Let $\phi$ be a nonstandard involution. Then,\\
$
(1) ~(\alpha A+\beta B)_{\phi}=A_{\phi}\phi(\alpha)+B_{\phi}\phi(\beta),~\alpha,\beta\in \mathbb{H},~A,B\in\mathbb{H}^{m\times n}.
$\\
$
(2) ~( A\alpha+ B\beta)_{\phi}=\phi(\alpha)A_{\phi}+\phi(\beta)B_{\phi},~\alpha,\beta\in \mathbb{H},~A,B\in\mathbb{H}^{m\times n}.
$\\
$
(3) ~(AB)_{\phi}=B_{\phi}A_{\phi},~A\in\mathbb{H}^{m\times n},~B\in\mathbb{H}^{n\times p}.
$\\
$
(4) ~(A_{\phi})_{\phi}=A,~A\in\mathbb{H}^{m\times n}.
$\\
$(5)$ If $A\in\mathbb{H}^{n\times n}$ is invertible, then $(A_{\phi})^{-1}=(A^{-1})_{\phi}.$
\\
$(6)~r(A)=r(A_{\phi}),A\in\mathbb{H}^{m\times n}.$\\
$(7) ~I_{\phi}=I,~0_{\phi}=0.$\\
$(8)~  (A_{\phi})^{\dag}=(A^{\dag})_{\phi},$ \cite{heela}.  \\
$(9)~   (L_{A})_{\phi}=R_{A_{\phi}},~ (R_{A})_{\phi}=L_{A_{\phi}},$ \cite{heela}.
\end{property}

For more properties of $\phi$-Hermitian quaternion matrix, we refer the reader to the recent book \cite{rodman} and the papers \cite{heaaca} and \cite{heela}. The following theorem gives some solvability conditions to the system (\ref{june16equ051}) in terms of ranks.

\begin{theorem}\label{theorem41}
The system (\ref{june16equ051})  is consistent if and only if
\begin{align}\label{july27equ001}
r\begin{pmatrix}E_{i}&A_{i}&C_{i}&F_{i}\\(A_{i})_{\phi}&0&0&0\end{pmatrix}=r\begin{pmatrix}A_{i}&C_{i}&F_{i}\end{pmatrix}+r(A_{i}),
\end{align}
\begin{align}\label{july27equ002}
r\begin{pmatrix}E_{i}&A_{i}&C_{i}\\(A_i)_{\phi}&0&0\\(F_{i})_{\phi}&0&0\end{pmatrix}=r\begin{pmatrix}A_{i}&C_{i}\end{pmatrix}+r\begin{pmatrix}A_{i}\\F_{i}\end{pmatrix},
\end{align}
\begin{align}\label{july27equ003}
&r\begin{pmatrix}
\begin{smallmatrix}
C_{m}&E_{m}&F_{m}& & & & &A_{m}& & & \\
&(F_{m})_{\phi}& &(C_{m+1})_{\phi}& & & & & & & \\
& &C_{m+1}&-E_{m+1}&(F_{m+1})_{\phi}& & & &A_{m+1}& & \\
& & &\ddots&\ddots&\ddots & & & &\ddots& \\
& & & &C_{n}&(-1)^{n-m}E_{n}&(F_{n})_{\phi}& & & &A_{n}\\
&(A_{m})_{\phi}& & & & & & & & & \\
& & &(A_{m})_{\phi}& & & & & & & \\
& & & &\ddots & & & & & & \\
& & & & &(A_{n})_{\phi}& & & & & 
\end{smallmatrix}
\end{pmatrix}\nonumber\\
&=r\begin{pmatrix}
\begin{smallmatrix}
C_{m}&F_{m}& & &A_{m}& & & \\
&C_{m+1}&F_{m+1}& & &A_{m+1}& & \\
& &\ddots&\ddots & & &\ddots & \\
& & &C_n &F_n& & &A_n
\end{smallmatrix}
\end{pmatrix}
+r\begin{pmatrix}
\begin{smallmatrix}
F_{m}&C_{m+1}& & & \\
&F_{m+1}&C_{m+2}& & \\
& &\ddots&\ddots& \\
& & &F_{n-1}&C_{n}\\
A_{m}& & & & \\
&A_{m+1}& & & \\
& & &\ddots& \\
& & & &A_{n}
\end{smallmatrix}
\end{pmatrix},
\end{align}

\begin{align}\label{july27equ004}
&r\begin{pmatrix}
\begin{smallmatrix}
(C_{m})_{\phi}& & & & & & & & & \\
E_{m}&F_{m}& & & & &A_m& & & \\
(-F_{m})_{\phi}& &(C_{m+1})_{\phi}& & & & & & & \\
&C_{m+1}&-E_{m+1}&F_{m+1}& & &&A_{m+1}&&\\
& &(F_{m+1})_{\phi}& &\ddots && & &\ddots &\\
& & &C_{m+2}&\ddots &(C_n)_{\phi}& & &&A_n\\
& & & &\ddots &(-1)^{n-m}E_n& & &&\\
& & & & &(F_n)_{\phi}& & &&\\
(A_{m})_{\phi}& & & & & & & &&\\
& &(A_{m+1})_{\phi }& & & & & &&\\
& & &  &\ddots & & & &&\\
& & &  & &(A_n)_{\phi }& & &&
\end{smallmatrix}
\end{pmatrix}\nonumber\\
&=r\begin{pmatrix}
\begin{smallmatrix}
F_{m}& & & &A_m&&\\
C_{m+1}&F_{m+1}& & & &A_{m+1}&\\
&C_{m+2}&\ddots&F_{n-1}& &&\ddots \\
& &\ddots&C_n& &&&A_n
\end{smallmatrix}
\end{pmatrix}
+r\begin{pmatrix}
\begin{smallmatrix}
C_{m}& & \\
F_{m}&C_{m+1}& \\
&F_{m+1}&\ddots&\\
& &\ddots &C_n\\
& & &F_n\\
A_{m}& & \\
&A_{m+1}& \\
& &\ddots\\
& & &A_n
\end{smallmatrix}
\end{pmatrix},
\end{align}
\end{theorem}

\begin{proof}
We prove that the system (\ref{june16equ051}) has a solution if and only if the system of quaternion matrix equations
\begin{align}\label{june16equ510}
  \left\{\begin{array}{c}
A_{1}\widehat{X}_{1}+\widehat{Y}_{1}(A_{1})_{\phi}+C_{1}\widehat{Z}_{1}(C_{1})_{\phi}+F_{1}\widehat{Z}_{2}(F_{1})_{\phi}=E_{1},\\
A_{2}\widehat{X}_{2}+\widehat{Y}_{2}(A_{2})_{\phi}+C_{2}\widehat{Z}_{2}(C_{2})_{\phi}+F_{2}\widehat{Z}_{3}(F_{2})_{\phi}=E_{2},\\
\vdots\\
A_{k}\widehat{X}_{k}+\widehat{Y}_{k}(A_{k})_{\phi}+C_{k}\widehat{Z}_{k}(C_{k})_{\phi}+F_{k}\widehat{Z}_{k+1}(F_{k})_{\phi}=E_{k},
\end{array}
  \right.
\end{align}
has a solution. If the system (\ref{june16equ051}) has a  solution, it can be expressed as
\begin{align*}
X_i=\dfrac{\widehat{X}_{i}+(\widehat{Y}_{i})_{\phi}}{2},
Z_i=\dfrac{\widehat{Z}_{i}+(\widehat{Z}_{i})_{\phi}}{2},
Z_{i+1}=\dfrac{\widehat{Z}_{i+1}+(\widehat{Z}_{i+1})_{\phi}}{2},~(1\le i\le k).
\end{align*}
We can give the
solvability conditions to the system  (\ref{mainsystem02}) by
Theorem \ref{theorem31}.
\end{proof}

\section{\textbf{Conclusions}}

We have provided some necessary and sufficient conditions for the existence of a solution to the system of quaternion matrix equations (\ref{mainsystem01}) in terms of ranks. Based on the results of the system (\ref{mainsystem01}), we have presented some necessary and sufficient conditions for the existence of a solution to the system of quaternion matrix equations involving $\phi$-Hermicity (\ref{mainsystem02}). Some known results can be viewed as special cases of the ones obtained in this paper.

%\section{\textbf{Acknowledgement}}
%
%The author would like to thank the anonymous
%referee for careful reading of the manuscript and valuable
%suggestions.

\end{document}